\documentclass[12pt, reqno]{amsart}
\usepackage{amscd}
\usepackage{bbm}
\usepackage{dsfont}
\usepackage{enumerate}
\usepackage{latexsym}
\usepackage{srcltx}
\usepackage{amsfonts}
\usepackage{amssymb}
\usepackage{datetime}
\usepackage{setspace}
\usepackage{enumerate}
\usepackage{amsthm}
\usepackage{graphicx}
\usepackage{epstopdf}

\newtheorem{theorem}{Theorem}[section]
\newtheorem{proposition}[theorem]{Proposition}
\newtheorem{lemma}[theorem]{Lemma}
\newtheorem{corollary}[theorem]{Corollary}

\theoremstyle{definition}
\newtheorem{example}[theorem]{Example}
\newtheorem{definition}[theorem]{Definition}

\newtheorem{remark} [theorem] {Remark}
\newtheorem{conjecture} [theorem] {Conjecture}

\begin{document}
\title{Spectrum of weighted composition operators. Part XI. The essential spectra of some weighted composition operators on the disc algebra.}

\author{Arkady Kitover}

\address{Community College of Philadelphia, 1700 Spring Garden St., Philadelphia, PA, USA}

\email{akitover@ccp.edu}

\author{Mehmet Orhon}

\address{University of New Hampshire, 105 Main Street
Durham, NH 03824}

\email{mo@unh.edu}


\date{\today}

\begin{abstract}
  We obtain a complete description of semi-Fredholm spectra of operators of the form $(Tf)(z) = w(z)f(B(z)$ acting on the disc algebra in the case when $B$ is either elliptic or double parabolic finite Blaschke product of degree $d \geq 2$, and $w$ has no zeros on the unit circle. In the case when $B$ has zeros on the unit circle we state a conjecture that would provide a description of the essential spectra in this case, but sofar is proved only in one direction. Our results hint on the possibility of interesting connections between the spectral properties of weighted composition operators and complex dynamics.
\end{abstract}

\maketitle

\markboth{Arkady Kitover and Mehmet Orhon}{Spectrum of weighted composition operators.}

\section{Introduction} 
Weighted composition operators on Banach spaces of analytic functions remain the subject of intensive research. While the appearance of the excellent books~\cite{Sha}, ~\cite{Cow}, ~\cite{Bo}, and~\cite{Lef} considerably enriched our knowledge of this subject, many challenging and interesting questions remain unsolved. One of the reasons of the continuing interest in the properties of composition and weighted composition operators is their intrinsic connection to topological and complex dynamics and that was our main motivation when writing this paper.

In~\cite{KO} the authors investigated some properties of essential spectra of weighted composition operators on uniform algebras. As a consequence we obtained in~\cite{KO} a description of essential spectra of weighted automorphisms of the disc algebra, and in~\cite{KO1} we extended these results to weighted automorphisms of the polydisc algebras. 

The next logical step is to investigate the essential spectra of weighted isometries of the disc algebra, i.e. of the operators of the form $(Tf)(z) = w(z)f(B(z))$, where $f$ and $w$ belong to the disc algebra and $B$ is a finite Blaschke product of degree $d \geq 2$.

In Theorem~\ref{t4} we provide a full description of the essential spectra of the weighted isometries of the disc algebra in the case when the Blaschke product $B$ is either elliptic or double parabolic (see the definition below) and the weight $w$ does not have zeros on the unit circle.

The case when $w$ has zeros on the unit circle is considerably more complicated. The results related to this case are contained in Theorems~\ref{t6} and~\ref{t11}, see also Examples~\ref{e6} and~\ref{e7} and Conjecture~\ref{co1}.

\section{Preliminaries and axillary results} 

Let us briefly recall the basic notions and notations used below.  All the linear spaces are considered over the field $\mathds{C}$ of all complex numbers.

\noindent Let $T$ be a continuous linear operator on a Banach space $X$. We denote the spectrum of the operator $T$ as $\sigma(T)$.

\noindent Recall that the approximate point spectrum of $T$ is defined as 
 
\begin{equation} \label{eq1}
\begin{split}
 & \sigma_{a.p.}(T) = \{\lambda \in \mathds{C} : \; \text{there are} \; x_n \in X \; \text{such that} \\
 &  \|x_n\|=1 \; \text{and} \; Tx_n - \lambda x_n \rightarrow 0\},
 \end{split}
\end{equation}

Recall that a sequence $x_n \in X$ is called singular if it does not contain any norm-convergent subsequence.

 The upper and lower semi-Fredholm spectra $\sigma_{usf}(T)$ and $\sigma_{lsf}(T)$ of $T$ are defined as follows (see \cite{EE}):

\begin{equation} \label{eq2}
\begin{split}
& \lambda \in \sigma_{usf}(T) \Leftrightarrow \lambda \in \sigma_{a.p.}(T) \; \text{and the sequence} \; x_n \in X \; \text{from} \; (1) \\
& \text{ can be chosen in such a way that it is singular}.
\end{split}
\end{equation}

\begin{equation} \label{eq3}
  \sigma_{lsf}(T) = \sigma_{usf}(T^\prime),
\end{equation}
where $T^\prime$ is the Banach adjoint of $T$.

By $\mathds{A}$ we denote the disk algebra, the Banach algebra of all the functions analytic in the open unit disk $\mathds{U}$ and continuous in the closed unit disk $\mathds{D}$ endowed with the supremum norm.

Let us recall that a finite Blaschke product of degree $d$ is defined as
\begin{equation} \label{eq4}
  B(z) =e^{i\theta} \prod \limits_{j=1}^d \frac{z-a_j}{1-\overline{a_j}}, \; |a_j| < 1.
\end{equation}
We will always assume that $d > 1$.

We need the following classification of finite Blaschke products. See e.g.~\cite[Definition 3.2, p. 25]{Co}.

Let $B$ be a finite Blaschke product of degree $d > 1$. Let $z_0 \in \mathds{D}$ be the Wolff-Denjoy point of $B$ (i.e. the unique point in $\mathds{D}$ such that the iterations of $B$ converge to $z_0$ uniformly on compact subsets of $\mathds{U}$). We say that $B$ is:

\textbf{elliptic} if $z_0 \in \mathds{U}$,

\textbf{hyperbolic} if $z_0 \in \mathds{T} = \partial \mathds{D}$ and $0 < |B^\prime(z_0)| < 1$,

\textbf{single parabolic} if $z_0 \in \mathds{T}$, $|B^\prime(z_0)| = 1$, and $z_0$ has multiplicity 3,

\textbf{doubly parabolic} if $z_0 \in \mathds{D}$, $|B^\prime(z_0)| = 1$, and $z_0$ has multiplicity 2. 

\bigskip

We need two deep results from topological dynamics. The first of them is due to M. Shub.

   \begin{theorem} \label{t2} (see~\cite{Sh} and~\cite[Theorem 4.4, p.32]{Co})
Let $B$ be a finite Blaschke product of degree $d \geq 2$. The map $B : \partial \mathds{D} \rightarrow \partial \mathds{D}$ is topologically conjugate to $z^d$ if and only if $B$ is either elliptic or double parabolic.
\end{theorem}

To state the second result we introduce the following notation. Let $d > 1$ and let $t \in \partial \mathds{D}$ be a periodic point of the map $\varphi(z) = z^d$ of the smallest period $m$. Then the ergodic $\varphi$-invariant measure $\mu_t$ is defined as
\begin{equation}\label{eq20}
  \mu_t = \sum \limits_{j=0}^{m-1} \delta_{\varphi^j(t)}.
\end{equation}
We will call the ergodic measures defined by the formula~(\ref{eq20}) periodic ergodic measures. The following theorem follows from~\cite[Proposition 21.3, p.194]{DGS} and from~\cite[Theorem 45]{KLO}.

\begin{theorem} \label{t3} Let $\varphi(z) = z^d$, where $d > 1$. The set of all periodic ergodic measures is dense in the set of all $\varphi$-invariant ergodic measures in the topology 
$\sigma(C(\partial \mathds{D}),C(\partial \mathds{D})^\prime)$.
\end{theorem}

The topological dynamics of the maps $z \rightarrow z^d, z \in \partial \mathds{D}$ was studied in seminal papers of Bullett and Sentenac, ~\cite{BS}, Blokh et al.,~\cite{BMMOP}, and Malaugh, ~\cite{Ma}.

For our purposes we need only a tiny part of the important results obtained in these papers (see Theorem~\ref{t9}). First we need the following definition given in~\cite[p. 456]{BS}.

\begin{definition} \label{d1}
  Let $A$ be a subset of the unit circle $\mathds{T}$ and $f$ be a map $A \rightarrow \mathds{T}$. We say that
  $f$ is order-preserving if for each triple $(a,b,c)$ in $A$ the triple $(f(a),f(b),f(c))$ lies in the same cyclic order around $\mathds{T}$, or else is degenerate (i.e. two or all three, of the points coincide).
\end{definition}

\begin{theorem} \label{t9} See~\cite{BS,BMMOP}.
  Let $d \geq 2$. There are $z^d$-periodic points $z_0$ and $z_i \in \mathds{T}, i \in \mathds{N}$ such that
  \begin{enumerate}
    \item $Tr(z_i) \cap Tr(z_j) = \emptyset, i,j \in \mathds{N}, i \neq j$, where $Tr(z_i)$ is the finite trajectory of $z_j$ under the action of the map $z \rightarrow z^d$.
    \item The map $z \rightarrow z^d$ is order preserving on $Tr(z_i), i \in \mathds{N}$.
    \item The map $z \rightarrow z^d$ is not order preserving on $Tr(z_0)$.
  \end{enumerate}
\end{theorem}

 The following theorem follows from Theorem 2.1 and Corollary 2.3 in~\cite{KO}.

\begin{theorem} \label{t1} Let $(Tf)(z) = w(z)f(B(z), f \in \mathds{A}, z \in \mathds{D}$, where $B$ is a finite Blaschke product of degree $n \geq 2$ and $w \in \mathds{A}$.

Then, $\sigma_{usf}(T) = \sigma_{a.p.}(T)$ and the sets $\sigma(T)$ and $\sigma_{usf}(T)$ are rotation invariant.
\end{theorem}

In the following lemma, which might be of some independent interest, we use the notations
$\varphi^1 = \varphi$, $\varphi^n = \varphi (\varphi^{n-1}), n \geq 2$,
$w_1 = w$,  and $w_n = w(w \circ \varphi)...(w \circ \varphi^{n-1}), n \geq 2$.

\begin{lemma} \label{l1}   Let $K$ be a compact Hausdorff space and $\varphi$ be an open continuous map of $K$ onto itself. Let $w \in C(K)$ and $T = wT_\varphi$. Let $\lambda \in \mathds{C}, \lambda \neq 0$ and $\lambda \in \sigma_{a.p.}(T)$.

The following conditions are equivalent.

(1) $\lambda \in \sigma_{usf}(T) \setminus \{0\}$.

 (2) There is a $k \in \partial \mathds{D}$ such that for any $m, n \geq 0$ and for any $e \in \partial \mathds{D}$ we have
  
  \begin{equation}\label{eq5}
    \varphi^m(e) = \varphi^n(k) \Rightarrow \Big{|}\frac{w_m(e)}{\lambda^m}\Big{|} \leq \Big{|}\frac{w_n(k)}{\lambda^n}\Big{|}.
  \end{equation}
 \end{lemma}
 
 \begin{proof} $(1) \Rightarrow (2)$
 
  Assume first that the compact space $K$ is metrizable. Without loss of generality we can assume that $\lambda = 1$. Let $f_p \in C(K), p \in \mathds{N}$, be such that $\|f_p\| =1$ and $Tf_p - f_p \rightarrow 0$. Let $k_p \in K$ be such that $|f_p(k_p)|=1$. We can assume without loss of generality that
  the sequence $k_p$ converges to a $k \in K$.
  
  Let $e \in K$ and $m,n \geq 0$ be such that $\varphi^m(e) = \varphi^n(k)$. It is immediate to see that because the map $\varphi$ is open, there is a subsequence $k_{p_q}$ of $k_p$ and the points $e_q \in K$ such that the sequence $e_q$ converges to $e$ and $\varphi^m(e_q) = \varphi^n(k_{p_q})$.
  
  Next notice that
  
  \begin{equation}\label{eq21}
        \begin{split}
     & T^mf_{p_q}(e_q) = w_m(e_q)f_{p_q}(\varphi^m(e_q)) - f_{p_q}(e_q) \rightarrow 0, \\
     & T^nf_{p_q}(k_{p_q}) = w_n(k_{p_q})f_{p_q}(\varphi^n(k_{p_q})) - f_{p_q}(k_{p_q}) \rightarrow 0.
    \end{split}
  \end{equation}
  Recalling that $\varphi^m(e_q) = \varphi^n(k_{p_q})$ and that $|f_{p_q}(e_q)| \leq |f_{p_q}(k_{p_q})|=1$ we see that it follows from~(\ref{eq21}) that $|w_m(e)| \leq |w_n(k)|$, and~(\ref{eq5}) is proved.
  
   The case of an arbitrary compact Hausdorff space $K$ follows easily. Indeed, it is sufficient to consider the separable subalgebra of $C(K)$ generated by the functions $T^if_n, n \in \mathds{N}, i \geq 0$, and their complex conjugates.
   
   $(2) \Rightarrow (1)$ We can assume without loss of generality that $\lambda =1$. Assume first that the point $k$ is not eventually $\varphi$-periodic. Fix an $n \in \mathds{N}$ and let $u = \varphi^n(k)$. It follows 
  from~(\ref{eq5}) that there is an open neighborhood $U$ of $u$ such that
  
  \begin{equation}\label{eq9}
    cl\varphi^{-i}(U) \cap cl\varphi^{-j}(U) = \emptyset, \; 0 \leq i < j \leq 2n+1
  \end{equation}
  and
    \begin{equation}\label{eq10}
    |w_{2n+1}(v)| \leq 2|w_n(k)|, \; v \in \varphi^{-2n-1}(U).
  \end{equation}
  
   Let $f \in C(K)$ be such that 
  $f(u) = \|f\| =1$ and $f \equiv 0$ on $K \setminus U$. Let $g = \frac{1}{w_n(k)}f$ and let
  
  \begin{equation}\label{eq7}
   G = \sum \limits_{j = 0}^{2n} \Big{(} 1 - \frac{1}{\sqrt{n}} \Big{)}^{|n - j|} T^jg
  \end{equation}
   Then,
   
   \begin{equation}\label{eq8}
   \begin{split}
    & TG - G = -\Big{(} 1 - \frac{1}{\sqrt{n}} \Big{)}^ng - \frac{1}{\sqrt{n}}\sum \limits_{j=1}^n \Big{(} 1 - \frac{1}{\sqrt{n}} \Big{)}^{|n - j|} T^jg  \\
      & + \frac{1}{\sqrt{n}} \sum \limits_{j= n+1}^{2n} \Big{(} 1 - \frac{1}{\sqrt{n}} \Big{)}^{|n - j|} T^jg + \Big{(} 1 - \frac{1}{\sqrt{n}} \Big{)}^nT^{2n+1}g.\\
    \end{split}
   \end{equation}
   Notice that $\|G\| \geq 1$. Indeed, it follows from~(\ref{eq9})  and~(\ref{eq7}) that $G(k) = 1$. Next, from~(\ref{eq8}) and~(\ref{eq9})\ we have
   
   \begin{equation}\label{eq11}
     \|TG - G\| \leq \Big{(} 1 - \frac{1}{\sqrt{n}} \Big{)}^n + \frac{1}{\sqrt{n}}\|G\| + \Big{(} 1 - \frac{1}{\sqrt{n}} \Big{)}^n \|T^{2n+1}g\|.
   \end{equation}
   It remains to notice that
   
   \begin{equation}\label{eq12}
   T^{2n+1}g(s) =  \left\{
       \begin{array}{ll}
         w_{2n+1}(s)g(\varphi^{2n+1}(s)), & \hbox{if $s \in \varphi^{-2n-1}(U)$;} \\
         0, & \hbox{otherwise.}
       \end{array}
     \right.
   \end{equation}
Finally, from~(\ref{eq11}),~(\ref{eq12}), and~(\ref{eq10}) we get
\begin{equation}\label{eq13}
  \|TG - G\| \leq \Big{(}  3\Big{(} 1 - \frac{1}{\sqrt{n}} \Big{)}^n + \frac{1}{\sqrt{n}} \Big{)}\|G\|.
\end{equation}
Because $n$ is arbitrary large, $1 \in \sigma_{a.p.}(T)$.

In the case when the point $k$ is eventually periodic we fix an $n \in \mathds{N}$. The conditions of the lemma allow us to choose a point $k_n$ and its open neighborhood $U_n$ such that~(\ref{eq9}) and~(\ref{eq10}) are satisfied, and $|w_n(k_n)| \geq 1/2$. Then we proceed as in the previous part of the proof.
 \end{proof}  
 
  As the following simple example shows the statement of Lemma~\ref{l1} is in general not true without the assumption that the map $\varphi$ is open.
    
  \begin{example} \label{e1} Let $K = \{0\} \cup \{\frac{1}{n}, n \in \mathds{N}\} \cup [1,2]$.
  
  Let $\varphi(0) = 0, \varphi(\frac{1}{n}) = \frac{1}{n-1}, n=2,3, \ldots$, and $\varphi(x) = x, x \in [1,2]$.
  
  Let $w(x) =x, x \in [1,2]$ and $w(k) = 2, k \in K \setminus [1,2]$.
  
  Let $T = wT_\varphi$ be the corresponding weighted composition operator on $C(K)$. It is immediate to see that $1 \in \sigma_{a.p.}(T)$, but at no point $k \in K$ the inequalities $|w_n(k)| \geq 1$ and $t \in \varphi^{-n}(\{k\}) \Rightarrow |w_n(t)| \leq 1$ are satisfied.    
  \end{example} 
  
  We will need the following corollary of Lemma~\ref{l1}.
    
  \begin{corollary} \label{c1}
    Let $K$ be a compact Hausdorff space, $\varphi$ be a continuous open map of $K$ onto itself, $w \in C(K)$,
    and $(Tf)(k) = w(k)f(\varphi(k)), f \in C(K), k \in K$. Let $\lambda \in \mathds{C} \setminus \{0\}$ and assume that there is a point $k \in K$ such that
    \begin{enumerate}
      \item $|w_n(k)| \geq |\lambda|^n, n \in \mathds{N}$;
      \item if $s \in \varphi^{-1}(k)$ and the point $s$ is not $\varphi$-periodic than $w(s) =0$.
          \end{enumerate}
          Then $\lambda \in \sigma_{a.p.}(T)$.
  \end{corollary}
  
  The next simple lemma complements the statements of Lemma~\ref{l1} and Corollary~\ref{c1}.

  \begin{lemma} \label{l3} Let $K$ be a metrizable Hausdorff compact space, $\varphi$ be a continuous open map of $K$ onto itself, $w \in C(K)$, and $T$ be the weighted isometry on $C(K)$ defined as
  
  \begin{equation*}
    (Tf)(k) = w(k)f(\varphi(k)), f \in C(K), k \in K.
  \end{equation*}
  The following conditions are equivalent
  
  $(1) \; 0 \in \sigma_{a.p.}(T)$.
  
  $(2)$ There is a $k \in K$ such that $w \equiv 0$ on $\varphi^{-1}(\{k\})$.    
  \end{lemma}

  \begin{proof}
    $(1) \Rightarrow (2)$. Let $f_n \in C(K), \|f_n\| = f_n(k_n) =1$ and $Tf_n \rightarrow 0$. Let 
    $u \in \varphi^{-1}(\{k_n\}$ then $(Tf_n)(u) = w(u)$ and therefore 
        
    \begin{equation}\label{eq25}
      \max \limits_{v \in \varphi^{-1}(\{k_n\})} |w(v)| \mathop \rightarrow \limits_{n \to \infty} 0.
    \end{equation}
    We can assume without loss of generality that $k_n \rightarrow k \in K$.
    It follows from~(\ref{eq25}) and the fact that $\varphi$ is open that $w \equiv 0$ on $\varphi^{-1}(\{k\})$.
    
    \noindent $(2) \Rightarrow (1)$. Assume (2) and assume that $d$ is a metric on $K$. Let
    $F_n = \{t \in K : \; d(t,k) \leq 1/n, n \in \mathds{N}\}$. Let $f_n \in C(K)$ be such that 
    $\|f_n\| = 1 = f_n(k)$ and $supp f_n \subseteq F_n$. Then $Tf_n \rightarrow 0$. 
  \end{proof}

We also need the following lemma.

\begin{lemma} \label{l2} Let $\varphi(z) = z^d, d \geq 2$ and let $u,v \in \partial \mathds{D}$. Then there is a sequence
$w_n$ such that $w_1 = u, \varphi(w_n) = w_{n-1}, n \geq 2$, and $w_n \rightarrow v$.
\end{lemma}

\begin{proof} The statement of the lemma becomes obvious if we recall that the map $z \rightarrow z^d, z \in C(\partial \mathds{D})$ is topologically equivalent to the map $t \rightarrow dt, t \in \mathds{R}/\mathds{Z}$ (mod 1). Indeed, if we identify an $t \in [0,1)$ with the sequence $\{ \alpha_i, i \in \mathds{N}\}$, where 
$\alpha_i \in \{0,1, \ldots d-1\}$ and $t = \sum \limits_{i=1}^\infty \alpha_i/d^i$ (with the usual agreement that the rational numbers are represented in the unique way), then the map $t \rightarrow dt$ is the backward shift on the space of these sequences.
\end{proof}

We conclude this section by an important result due to Globevnik (see~\cite[p.1]{Gl})

\begin{theorem} \label{t10} Let $F \subset \mathds{T}$ be a closed set of Lebesgue measure zero and let $p$ be a positive continuous function on $\mathds{T}$ such that the conjugate of $\log{p}$ is continuous on $\mathds{T} \setminus F$. Let $f \in C(F)$ satisfy $|f(s)| \leq p(s), s \in F$.

Given a neighbourhood $U \subset \mathds{T}$ of $F$ there is an extension $\tilde{f} \in \mathds{A}$ of $f$
such that
\begin{enumerate}[(i)]
  \item $|\tilde{f}(z)| \leq p(z), z \in U$.
  \item $|\tilde{f}(z)| = p(z), z \in \mathds{T} \setminus U$.
\end{enumerate}
  \end{theorem}

  \begin{remark} \label{r1}
  In particular, the statement of Theorem~\ref{t10} holds if $p$ is continuously differentiable on $\mathds{T} \setminus F$.  
  \end{remark}

\section{The main results}

\begin{theorem} \label{t4}
 Let $(Tf)(z) = w(z)f(B(z), f \in \mathds{A}, z \in \mathds{D}$, where $B$ is a finite Blaschke product of degree $d \geq 2$ and $w \in \mathds{A}, w \neq0$.
  
  Assume that
  \begin{enumerate}
    \item The Blaschke product $B$ is either elliptic or double parabolic.
    \item The weight $w$ has no zeros on the unit circle $\mathds{T}$.
  \end{enumerate}
  Then the spectral radius $\rho(T) >0$, $\sigma(T) = \sigma_{lsf}(T) = \rho(T)\mathds{D}$, 
  and $\sigma_{usf}(T) = \rho(T)\mathds{T}$.
  \end{theorem}

  \begin{proof} In virtue of Theorem~\ref{t2} we can assume without loss of generality that $B(z) = z^d, d \in \mathds{N}, d \geq 2$.
  The inequality $\rho(T) > 0$ follows from the well known formula (see e.g.~\cite{Kit})
    
    \begin{equation}\label{eq22}
      \rho(T) = \max \limits_{\mu \in M_{z^p}} \exp \int \ln{|w|} d \mu, 
    \end{equation}
    where $M_{z^d}$ is the set of all $z^d$-invariant probability measures in $C(\mathds{T})^\prime$, and from the standard fact that for $w \in \mathds{A}, w \neq 0$ we have 
  
  \begin{equation*}
    \int \limits_0^{2\pi} {\ln{|w(e^{i\theta})d\theta}} > -\infty.
  \end{equation*}
    To prove that $\sigma_{usf}(T) = \sigma_{a.p.}(T) = \rho(T)\mathds{T}$ we assume to the contrary that there is a $\lambda \in \sigma_{a.p.}(T)$ such that $0 < |\lambda| < \rho(T)$. It follows from~(\ref{eq22}) and 
    Theorem~\ref{t3} that there is a $z^d$-periodic point $t \in \mathds{T}$ of the smallest period $m$ such that
    
    \begin{equation}\label{eq24}
       |w_m(t)| > |\lambda|^m. 
    \end{equation}
     By Lemma~\ref{l1} and Theorem~\ref{t1} there is a point $k \in \mathds{T}$ such that 
        \begin{equation}\label{eq23}
      u \in \varphi^{-n}(\{k\}) \Rightarrow |w_n(u)| \leq |\lambda|^n
      , n \in \mathds{N}
    \end{equation}
    But, by Lemma~\ref{l2} there is a sequence $k_n \in \mathds{T}$ such that $k_1 = k, \varphi(k_n) = k_{n-1}, n \geq 2$ and $k_n \rightarrow t$. Clearly, in virtue of the inequality~(\ref{eq24}), the inequality~(\ref{eq23}) cannot be satisfied. Thus, $\sigma_{usf}(T) = \rho(T)\mathds{T}$.
    
    It remains to notice that clearly $codim(T\mathds{A}) = \infty$ and therefore 
    $\sigma_{lsf}(T) = \sigma(T) = \rho(T)\mathds{D}$.
  \end{proof}
  
  In the case when the weight $w$ has zeros on the unit circle the situation becomes considerably more complicated. Below we present some results concerning this case. The first of this results represents a slight improvement of Theorem~\ref{t4}.

  \begin{theorem} \label{t5}
Let $(Tf)(z) = w(z)f(B(z), f \in \mathds{A}, z \in \mathds{D}$, where $B$ is a finite Blaschke product of degree $d \geq 2$ and $w \in \mathds{A}$.
  
  Assume that
  \begin{enumerate}
    \item The Blaschke product $B$ is either elliptic or double parabolic.
    \item For any $t_0 \in \mathds{T}$ there are a $p \in \mathds{N}$ and points $t_1, \ldots, t_p \in \mathds{T}$ such that $B(t_j) = t_{j-1}, j = 1, \ldots, p$, and $w \neq 0$ on the set
        $\{t_1, \ldots, t_p, \bigcup \limits_{n \in \mathds{N}} \varphi^{-n}\{(t_p)\}\}$.
  \end{enumerate}  
  Then the spectral radius $\rho(T) >0$, $\sigma(T) = \sigma_{lsf}(T) = \rho(T)\mathds{D}$, 
  and $\sigma_{usf}(T) = \rho(T)\mathds{T}$.  
  \end{theorem}

  \begin{proof}
    It follows from Lemma~\ref{l3} that $0 \notin \sigma_{a.p.}(T)$. The case $0 < |\lambda| < \rho(T)$ can be considered exactly like in the proof of Theorem~\ref{t4}.
  \end{proof}
 
The following example illustrates the statement of Theorem~\ref{t5}.

\begin{example} \label{e6}
 Let $d \in \mathds{N}, d \geq 2$ and $(Tf)(z) = \frac{1-z}{2}f(z^d), f \in A, z \in \mathds{D}$. 

\noindent (1) If $d$ is odd, then $\sigma_{lsf}(T) = \sigma(T) = \mathds{D}$ and $\sigma_{usf}(T) = \mathds{T}$.

\noindent (2) If $d = 2k, k \in \mathds{N}$, then $\sigma_{lsf}(T) = \sigma(T) = \sin{\Big{(}\frac{\pi k}{2k+1}\Big{)}} \mathds{D}$ and $\sigma_{usf}(T) = \sin{\Big{(}\frac{k}{2k+1}\Big{)}} \mathds{T}$. 
\end{example}

\begin{proof}
  (1) This case is trivial, indeed by Theorem~\ref{t5} we only need to verify that $\rho(T) = 1$ and it is obvious because $|w_n(-1)| = 1, n \in \mathds{N}$.

(2)By Theorem~\ref{t5} it is sufficient to prove that $\rho(T) = \sin{\Big{(}\frac{\pi k}{2k+1}\Big{)}}$.

The inequality $\rho(T) \geq \sin{\Big{(}\frac{\pi k}{2k+1}\Big{)}}$ is trivial. Indeed, put
 $z =\exp{\Big{(}\frac{i\pi k}{2k+1}\Big{)}}$. Then, $|w_n(z)| = \sin^n{\Big{(}\frac{\pi k}{2k+1}\Big{)}}$.

To prove that $\rho(T) \leq \sin{\Big{(}\frac{\pi k}{2k+1}\Big{)}}$ it is sufficient to prove that for any $n \in \mathds{N}$ 
\begin{equation}\label{eq26}
  \max \limits_{\theta \in [0, 2\pi)} |sin(\theta)sin(2k\theta) \ldots sin((2k)^n \theta)|\leq \sin^n{\Big{(}\frac{\pi k}{2k+1}\Big{)}}.
\end{equation}
In the case when $k=1$ the inequality~(\ref{eq26}) was proved in~\cite[pp.21 and 86]{Pu}. A minor modification of the proof in~\cite{Pu} is sufficient to prove~(\ref{eq26}) in general. We provide the details for the reader's convenience. Following~\cite{Pu} we divide the proof of~(\ref{eq26}) into three steps.

(a) It is immediate to see that

\begin{equation*}
  \max \limits_{\theta \in [0,2\pi)} |sin^{2k}(\theta)sin(2k\theta)| = \sin^{2k+1}{\Big{(}\frac{\pi k}{2k+1}\Big{)}}.
\end{equation*}

(b) The expression
\begin{equation*}
   P_n = |sin^{2k}(\theta){sin^{2k+1}(2k\theta) \ldots sin^{2k+1}((2k)^{n-1}\theta)}sin((2k)^n \theta )| 
\end{equation*}
 takes its maximum value at the point $\theta = \frac{\pi k}{2k+1}$. We prove (b) by induction. For $n=1$ (b) follows from (a). Assume that (b) is true for an $n \in \mathds{N}$ and consider the ratio 
\begin{equation*}
   \frac{P_{n+1}}{P_n} = |sin^{2k}((2k)^n \theta)sin((2k)^{n+1}\theta)|. 
 \end{equation*}
 By (a) this ratio takes its greatest value at $\theta = \frac{\pi k}{2k+1}$, and (b) is proved.

(c) It follows from (b) that 
\begin{equation*}
  |sin(\theta)sin(2k\theta) \ldots sin((2k)^n \theta)|^{2k+1} \leq \sin^{n(2k+1)}{\Big{(}\frac{\pi k}{2k+1}\Big{)}}
\end{equation*}
 and~(\ref{eq26}) is proved.
\end{proof}

As the next example shows, we cannot omit condition (2) in the statement of Theorem~\ref{t5}.

\begin{example} \label{e7}
 Let $(Tf)(z) = w(z)f(z^2), f \in \mathds{A}, z \in \mathds{D}$ and let $w \in \mathds{A}$ be such that
$w(\exp{(i2\pi/3)}) = w(\exp{(i4\pi/3)}) = \|w\| = 1$, $w(1) = 1/2$, $w(-1) = 0$, and $-1$ is the only zero of $w$ on $\mathds{T}$.

It follows from Corollary~\ref{c1} and the reasoning in the proof of Theorem~\ref{t4} that
\begin{equation*}
  \sigma(T) = \sigma_{lsf}(T) = \mathds{D} \; \text{and} \; \sigma_{usf}(T) = (1/2)\mathds{T} \cup \mathds{T}.
\end{equation*}
 \end{example}
 
 Example~\ref{e7} is a special case of the following theorem.

 \begin{theorem} \label{t6}
   Let $B$ be either elliptic or double parabolic finite Blaschke product of degree $d \geq 2$.
   Let $\lambda_1, \ldots, \lambda_m \in \mathds{C}$, $0 \leq |\lambda_1| < |\lambda_2| <  \ldots < |\lambda_m|$.
   
   Then there is a $w \in \mathds{A}$ such that for the operator $T$,
   
   \begin{equation*}
     (Tf)(z) = w(z)f(B(z)), f \in \mathds{A}, z \in \mathds{D},
   \end{equation*}
   we have $\sigma_{usf}(T) = \bigcup \limits_{j=1}^m \lambda_j\mathds{T}$.
 \end{theorem}

 \begin{proof}
   We can assume without loss of generality that $B(z) = z^d$. Let $t_j \in \mathds{T}, j = 1, \ldots , m$ be 
 $z^d$-periodic points. Let $p_j$ be the smallest period of the point $z_j$. We can assume that the sets
 $P_j = \{z_j^{d^k}, k = 0, \ldots , p_j -1\}$ are pairwise disjoint. Let $w \in \mathds{A}$ be such that 
   \begin{enumerate}
     \item $w \equiv \lambda_j$ on $P_j$,
     \item $w \equiv 0$ on $B^{-1}(P_j) \setminus P_j$,
     \item $|w| \leq |\lambda_m|$ on $\mathds{T}$.
   \end{enumerate}
   Let $Tf = w(f \circ B)$. The equality $\sigma_{usf}(T) = \bigcup \limits_{j=1}^m \lambda_j\mathds{T}$  follows from Theorem~\ref{t1}, Corollary~\ref{c1}, and Lemmas~\ref{l3} and~\ref{l2}. 
   \end{proof}
   
   In the case of elliptic or double parabolic finite Blaschke products we can considerably strengthen the statement of Theorem~\ref{t6}.

   \begin{theorem} \label{t11}
     Let $B$ be an elliptic or double parabolic finite Blaschke product. Let $\lambda_i, i \geq 0$, be positive real numbers such that $\lambda_0 > \lambda_1 > \ldots$ and $\lim \limits_{n \to \infty} \lambda_n =0$.
     
     Then, there is a $w \in \mathds{A}$ such that 
     
     \begin{equation}\label{eq27}
       \sigma_{usf}(T) = \{0\} \cup \bigcup \limits_{i \geq 0} \lambda_i \mathds{T},
     \end{equation}
     where $(Tf)(z) = w(z)f(B(z)), f \in \mathds{A}, z \in \mathds{D}$.
   \end{theorem}

   \begin{proof}
     We can assume without loss of generality that $B(z) = z^d, d \geq 2$. Let $z_i, i \geq 0$, be $z^d$-periodic points in $\mathds{T}$ with the properties described in the statement of Theorem~\ref{t9}.
     It follows from Definition~\ref{d1} and properties (2),(3) in the statement of Theorem~\ref{t9} that
     $Tr(z_0) \cap cl \bigcup \limits_{i \in \mathds{N}} Tr(z_i) = \emptyset$. Therefore,
     $mes(cl \bigcup \limits_{i \in \mathds{N}} Tr(z_i)) = 0$ where $mes$ is the normalized Lebesgue measure on $\mathds{T}$.
     
     Let $E = cl \bigcup \limits_{i \in \mathds{N}} Tr(z_i)$ and $F = E \cup B^{-1}(E)$. We define a function $f$ on $F$ as follows
     
     \begin{equation}\label{eq28}
      f(z)= \left\{
         \begin{array}{ll}
           \lambda_n, & \hbox{if $z \in Tr(z_n)$;} \\
           0, & \hbox{if $z \in F \setminus \bigcup \limits_{i \in \mathds{N}} Tr(z_i)$. }
         \end{array}
       \right.
     \end{equation}
     It is immediate to see that $f \in C(F)$.
     
     Let $N \in \mathds{N}$ be such that $1/N < dist(Tr(z_0), F)$. For any $n \geq N$ let 
     $V_n = \{z \in \mathds{T} : dist(z,F) < 1/n\}$. By Theorem~\ref{t10} and Remark~\ref{r1} for any $n \geq N$ there is $f_n \in \mathds{A}$ such that
     \begin{enumerate}[(a)]
       \item $f_n \equiv \frac{3}{4^{n-N+1}}f$ on $F$,
       \item $f_n \equiv \frac{3}{4^{n-N+1}}\lambda_0$ on $Tr(z_0)$,
       \item $\|f_n\| = \frac{3}{4^{n-N+1}}\lambda_0$,
       \item $|f_n| >0$ on $\mathds{T} \setminus V_n$,
       \item $|f_{n+1}| \leq \frac{1}{4}|f_n|$ on $\mathds{T} \setminus V_{n+1}$.
     \end{enumerate}
     Let $w = \sum \limits_{n=N}^\infty f_n$ and $(Tf)(z) = w(z)f(B(z)), f \in \mathds{A}, z \in \mathds{D}$.
     It is immediate to see that $w \equiv \lambda_n$ on $Tr(z_n)$, where $n \geq 0$ and that $|w| >0$ on $\mathds{T} \setminus F$. Moreover, $\rho(T) = \lambda_0$.
     
     In virtue of Corollary~\ref{c1} we have $\lambda_i \mathds{T} \subset \sigma_{usf}(T), i \in \mathds{N}$.
     
     It remains to notice that if $|\lambda| \notin \{0, |\lambda|_i, i \geq 0\}$ then in virtue of Lemma~\ref{l1} and the reasoning in the proof of Theorem~\ref{t4} we have $\lambda \notin \sigma_{a.p.}(T)$.      
     \end{proof}
   
   We believe that in the case when the finite Blaschke product is elliptic or double parabolic the converse to the statement of Theorem~\ref{t11} is true.
  
\begin{conjecture} \label{co1}
  Let $B$ be an elliptic or double parabolic Blaschke product of degree $d \geq 2$. Then
     the set $\{|\lambda| : \lambda \mathds{T} \subset \sigma_{usf}(T)\}$ is at most countable and its only accumulation point (if any) is 0.
 \end{conjecture}

In conclusion we would like to notice that the statement of Theorem~\ref{t4} becomes false in the case of hyperbolic or single parabolic Blaschke products. The next proposition follows from Lemma~\ref{l1}.

\begin{proposition} \label{p1}
  Let $B$ to be a hyperbolic or single parabolic finite Blaschke product, $z_0$ be the Wolff-Denjoy point of $B$, and $\mathcal{B}$ be the immediate basin of attraction for $z_0$. Let $w \in \mathds{A}$ be such that  $|w(z_0)|=1$, $w$ has no zeros on $\mathds{T}$, and $|w(z)| < 1$ for any $z \in \mathds{T}, z \neq z_0$. Let $(Tf)(z) = w(z)f(B(z), f \in \mathds{A}, z \in \mathds{D}$. Then, $\sigma_{usf}(T) \supseteq \{\lambda \in \mathds{C}: \max \limits_{\varsigma \in \partial \mathcal{B}} |w| \leq |\lambda| \leq 1\}$.
\end{proposition}


\begin{thebibliography} {99}
\bibitem{BMMOP} Blokh A., Malaugh J.M., Mayer J.C., Oversteegen L.G., and Parris D.K., Rotational subsets of the circle under $z^d$, Topology and its Applications, 153, 1540–1570  (2006). 
\bibitem{Bo} Bourdon P.S. and Shapiro J.H. Cyclic phenomena for composition operators, Memoirs of AMS, 596 (1997).
 \bibitem{BS} Bullett S. and Sentenac P., Ordered orbits of the shift, square roots, and the devil's staircase,
 Math. Proc. Camb. Phil. Soc., 115, 451 - 481 (1994).  
\bibitem{Co} Corbacho J.N., Dynamics of finite Blaschke products, Barcelona (2024).
\bibitem{Cow} Cowen C. and Maccluer B., Composition operators on spaces of analytical functions, CRC Press (1995).
\bibitem{DGS} Denker M., Grillenberger C., Sigmund K., Ergodic theory on compact spaces, Lecture Notes in Mathematics 527, Springer (1976).
\bibitem{EE} Edmunds D.E. and Evans W.D., Spectral theory and differential operators, Clarendon Press (1987).
\bibitem{Gl} Globevnik J., The modulus of Rudin-Carleson extensions, Monatshefte für Mathematik, 105,
47 - 58 (1988).
\bibitem{Kit} Kitover A., Spectrum of weighted composition operators: part 1. Weighted composition operators on $C(K)$ and uniform algebras, Positivity 15, 639–659  (2011), DOI 10.1007/s11117-010-0106-4.
\bibitem{KO} Kitover A. and Orhon M, Spectrum of weighted composition operators. Part IX. The spectrum and essential spectra of some weighted composition operators on uniform algebras, https://doi.org/10.48550/arXiv.2307.01112, accepted to ``Positivity''.
 \bibitem{KO1} Kitover, A., Orhon, M. Spectrum of weighted composition operators part X: the spectrum and essential spectra of weighted automorphisms of the polydisc algebra. Positivity 28, 58 (2024). https://doi.org/10.1007/s11117-024-01069-w
\bibitem{KLO} Kwietniak D., Lacka M., Oprocha P., A panorama of specification-like properties and their consequences, Contemporary Mathematics, 669, 155 - 186 (2016).
\bibitem{Lef}  Lefévre P., Li D., Queffélec H., and Rodríguez-Piazza L., Composition operators on Hardy-Orlicz spaces, Memoirs of AMS, 1974 (2010). 
  \bibitem{Ma} Malaugh J.M., Rotational sets of the circle under $z(d)$. University of Alabama at Birmingham theses and dissertattions (2003).   
 \bibitem{Pu} The William Lowell Putnam mathematical competition. Problems and Solutions, 1965 - 1984. The Mathematical Association of America (1985).  
 \bibitem{Sha} Shapiro J.H., Composition operators and classical function theory, Springer (1993).     
 \bibitem{Sh}  Shub M., Endomorphisms of compact differentiable manifolds. American Journal
of Mathematics, 91(1), 175–199 (1969). 
  \end{thebibliography}
\end{document}